\documentclass[reqno]{amsart}[12pt]
\usepackage{fullpage}
\usepackage{amsmath, amssymb, latexsym, amsthm, amsfonts, mathrsfs,
  amscd,color, bbm, enumitem}
\usepackage{scalerel,stackengine}
\stackMath

\newcommand{\et}{\mathrm{\acute{e}t}}

\newcommand\reallywidehat[1]{%

\savestack{\tmpbox}{\stretchto{%
  \scaleto{%
    \scalerel*[\widthof{\ensuremath{#1}}]{\kern-.6pt\bigwedge\kern-.6pt}%
    {\rule[-\textheight/2]{1ex}{\textheight}}
  }{\textheight}%
}{0.5ex}}%
\stackon[1pt]{#1}{\tmpbox}%
}
\numberwithin{equation}{section}

\theoremstyle{plain}
\newtheorem{theorem}{Theorem}
\newtheorem*{theorem*}{Theorem}
\newtheorem{lemma}[theorem]{Lemma}
\newtheorem{corollary}[theorem]{Corollary}
\newtheorem{proposition}[theorem]{Proposition}

\theoremstyle{remark}
\newtheorem*{remark}{Remark}

\theoremstyle{definition}
\newtheorem{definition}{Definition}

\newcommand{\Q}{{\mathbbm Q}}

\newcommand{\Z}{{\mathbbm Z}}
\newcommand{\C}{{\mathbbm C}}

\newcommand{\F}{{\mathbbm F}}

\newcommand{\cO}{{\mathcal O}}

\begin{document}

\title{Finiteness theorems for potentially  equivalent  Galois representations: extension of Faltings' finiteness criteria}

\author{Plawan Das}

\address{Chennai Mathematical Institute,  INDIA.}  
\email{d.plawan@gmail.com}

\author{C.~S.~Rajan}

\address{Tata Institute of Fundamental  Research, Homi Bhabha Road,
Bombay - 400 005, INDIA.}  \email{rajan@math.tifr.res.in}

\subjclass{Primary 11F80; Secondary 11G05, 11G10}

\begin{abstract} 
We study the relationship between potential equivalence and character theory; we observe that potential equivalence of a representation $\rho$ is determined by an equality of an $m$-power character $g\mapsto Tr(\rho(g^m))$ for some natural number $m$. Using this, we extend Faltings' finiteness criteria to determine the equivalence of two $\ell$-adic,  semisimple representations of the absolute Galois group of a number field, to the context of potential equivalence. 

We also discuss finiteness results for twist unramified representations.

\end{abstract}

\maketitle

\section{Introduction}
\subsection{Twisted good reduction}
Let $X$ be a smooth projective variety over a number field $K$ and $v$ be a finite prime in $K$. Let $K_v$ denote the completion of $K$ at $v$ and $\cO_v$ the ring of $v$-adic integers in $K_v$. The variety $X$ is said to have good reduction at $v$ if $X$ has a smooth model $\mathcal{X}_v$ over   $\cO_v$, i.e., 
the generic fibre $ \mathcal{X}_v\times_{\mathrm{Spec}(\cO_v)}\mathrm{Spec}({K_v})$ of $ \mathcal{X}_v$ is isomorphic to $X\times_{\mathrm{Spec}(K)}\mathrm{Spec}({K_v})$ and  $\mathcal{X}_v$ is a smooth projective variety over $\cO_v$. 

In (\cite{Fon}), Fontaine showed that there are no abelian varieties over $\Q$ with everywhere good reduction. It can be seen that if a smooth projective curve $C$ of genus at least one defined over $K$ has good reduction at $v$, then its Jacobian $J(C)$ also acquires good reduction at $v$. In particular, Fontaine's theorem implies that there are no smooth projective curves of genus at least one over $\Q$ which has everywhere good reduction. 

In a bid to construct higher genus curves with everywhere good reduction over number fields of small degree, Baaziz and Boxall in (\cite{BB}),  introduced the notion of \textit{twisted good reduction} of smooth projective varieties. A smooth projective variety $X$ defined over a number field  $K$ is said to have {\em twisted good reduction at a finite place $v$}, if there exists a smooth projective variety $X'$ defined over $K$ which has good reduction at $v$ and is a form of $X$, i.e., becomes isomorphic to $X$ over a finite extension $L$ of $K$. 

Baaziz and Boxall constructed infinitely many pairs $(K, C)$ consisting of a genus two curve $C$ defined over a  quadratic number field $K$ that has everywhere twisted good reduction. They also showed the finiteness of isomorphism classes of elliptic curves defined over a number field $K$ with twisted good reduction outside a finite set $S$ of places of $K$. The proof of this latter result follows the lines of the original argument due to Shafarevich (see \cite[Chapter IV]{Se1}) and rests on Siegel's lemma on transcendence. 

One of our aims here is to generalize this finiteness result to the context of $\ell$-adic representations. For this, we make use of Faltings' finiteness theorems for pure $\ell$-adic representations (\cite{D}, \cite[Chapter V, Section 2]{FW}). The finiteness criteria proved by Faltings, is that given a natural number $n$, a non-archimedean local field $F$ of characteristic zero and a finite set  $S$ of places of $K$, there exists a fixed finite set $T$ of finite places of $K$, disjoint from the  given set $S$,  such that the traces of the Frobenius conjugacy classes at the places belonging to $T$, serves to uniquely determine a 
continuous, semisimple representation from $G_K$ to $GL_n(F)$ that is unramified outside $S$. From this, assuming a purity hypothesis on the representations,  finiteness results are deduced. 

When $A$ is an abelian variety over $K$, Serre and Tate (\cite{ST}), 
extending the work of N\'eron, Ogg, and Shafarevich on elliptic curves, showed that  $A$ has good reduction at a finite place $v$ of $K$, if and only if the representation of $G_K$ on the Tate module $V_\ell(A)$ attached to $A$ is unramified at $v$, for all $\ell$ not equal to the residue characteristic at $v$. In particular, this allows us to recover the results of Baaziz and Boxall alluded to above; also generalize it to abelian varieties with `sufficiently large' algebraic monodromy groups.

\subsection{Potential equivalence}
In (\cite{PR2}), refined versions of strong multiplicity one theorem were considered for potentially equivalent $\ell$-adic representations of the absolute Galois group $G_K$. The refinement consists in weakening the hypothesis of the multiplicity one theorem by assuming only that the given representations are locally potentially equivalent at a set of places of positive upper density. 

One can consider other refinements of the strong multiplicity one theorem; for example, the analytic forms of the strong multiplicity one theorem, works with {\em finitely}  many local representations attached to the archimedean places and finite places with a bound on the cardinality of the residue field depending on the analytic conductors. 

The above considerations lead us  to consider  extending Faltings' theorem to the context of potentially equivalent representations. 
This is achieved by first observing that in our context, potential equivalence for fixed $K, n $ and $F$ is determined by the  $m$-trace function 
$\chi_{\rho}^{[m]}: g\mapsto \mbox{Tr}(\rho(g^m))$ attached to a semisimple representation $\rho:G_K\to GL_n(F)$, for some natural number $m$ depending on the parameters involved. Using Newton identities, we linearize the $m$-trace function and then appeal to Faltings' method to prove the required generalization.

\section{Potential equivalence and $m$-power characters}
In this section, we give a generalization of character theory to determine potential equivalence of representations. Various parts of this relationship have been proved earlier in (\cite{PR1}, \cite{PR2}). The exposition here possibly brings more clarity to the relationship between the two concepts. 

Two representations of a group are said to be {\em potentially equivalent} if they become isomorphic upon restricting to a subgroup of finite index. 

\begin{definition} Given a linear representation $\rho: \Gamma\to GL_n(F)$, define its {\em $m$-power character} (or $m$-character or $m$-trace) $\chi^{[m]}_{\rho}$ to be the function $g\mapsto \mbox{Tr}(\rho(g^m))$ for $g\in \Gamma$. 

Two representations of $\Gamma$ to $GL_n(F)$ are said to be {\em $m$-character equivalent} if their $m$-character functions coincide. 
\end{definition}
For what follows, we could have also worked with the semisimple parts of the Jordan decomposition of these elements in the above definition. 
When $m=1$, the $m$-trace of $\rho$ is the character of $\rho$, and we obtain the classical result that the character of a semisimple representation determines the representation. 

\begin{theorem}\label{thm:poteq-char0}
Suppose $\Gamma$ is an abstract group and $\rho_1, \rho_2$ are semisimple representations of $\Gamma$ to $GL_n(F)$, where $F$ is a field of characteristic zero. Then the following are equivalent: 
\begin{enumerate}
\item The representations $\rho_1, \rho_2$ are potentially equivalent. 
\item For some natural number $m$, the $m$-power characters are equal: 
\[\chi^{[m]}_{\rho_1}=\chi^{[m]}_{\rho_2}.\]
\end{enumerate}
\end{theorem}
\begin{proof}
$(1)\implies (2)$ is obvious. To prove the reverse implication, let $\rho:=\rho_1\times \rho_2$ be the product map from $\Gamma$ to $GL_n(F)\times GL_n(F)$. Denote by  $G$ the Zariski closure of the image $\rho(\Gamma)$ in $GL_n\times GL_n$. Let $p_1$ and $p_2$ denote respectively the first and the second projections of $G$ to $GL_n$. The restrictions of these projections to $\rho(\Gamma)$ give respectively the representations $\rho_1$ and $\rho_2$.  The function $g\mapsto \mbox{Tr}(g^m)$ is a rational function on $GL_n$. By Zariski density of  $\rho(\Gamma)$ in $G$,
\[ \chi^{[m]}_{p_1}(g) =\chi^{[m]}_{p_2}(g) \quad \forall g\in G.\]
Let  $G^0$ denote the connected component of identity in $G$. Since the $m$-th power map is surjective in a connected algebraic group, $\chi_{p_1}(g) = \chi_{p_2}(g)$ for all $g\in G^0$.
The representations being semisimple, it follows that $G^0$ is reductive. From the equality of characters, it follows that the  $p_1\simeq p_2$ restricted to $G^0$. Hence the representations  $\rho_1$ and $\rho_2$ are equivalent restricted to the subgroup $\Gamma^0=\rho^{-1}(\rho(\Gamma)\cap G^0(F))$ of finite index in $\Gamma$. 
\end{proof}

\subsection{Elementwise potential equivalence and local field valued representations} While considering potential equivalence, it is natural to consider an extension where we allow an equality of the $m$-power trace, but with $m$ depending on the element of the group. 
\begin{definition}
Two representations $\rho_1, \rho_2$ of $\Gamma$ to $GL_n(F)$ are said to be 
{\em elementwise potentially equivalent} if for any $g\in\Gamma$, there exists a natural number $m_g\geq 1$ such that the elements $\rho_1(g)^{m_g}$ and $\rho_2(g)^{m_g}$ are conjugate in $GL_n(F)$. 
\end{definition}

We now show that when the representations are valued in a non-archimedean local field $F$, then elementwise potential equivalence and potential equivalence are equivalent:

\begin{theorem}\label{theorem:eltwisepe}
Suppose $\Gamma$ is an abstract group and $\rho_1, \rho_2$ are semisimple representations of $\Gamma$ to $GL_n(F)$, where $F$ is a non-archimedean local field of characteristic zero. Then the following are equivalent: 
\begin{enumerate}
\item The representations $\rho_1, \rho_2$ are potentially equivalent. 
\item The representations $\rho_1, \rho_2$ are elementwise potentially equivalent. 
\item There exists a natural number $m$ depending only on $n$ and $F$, such that the $m$-power characters are equal: 
\[\chi^{[m]}_{\rho_1}=\chi^{[m]}_{\rho_2}.\]
\end{enumerate}
\end{theorem}
We need to only prove the implication $(2)\implies (3)$. This follows from the following lemma:
\begin{lemma}\label{lemma} 
	 There exists a positive integer $m$ depending only on $n$ and $F$ with the following property: suppose  $g_1,g_2$ are any two elements in $GL_n(F)$ and suppose there exists a natural number  $k\geq 1$ such that $g_1^k$ and $g_2^k$ are conjugate in $GL_n(F)$. Then, $g_1^m$ and $g_2^m$ are conjugate in $GL_n(F)$. 
\end{lemma} 

\begin{proof}
 Let $E$ be the extension of $F$ in a separable closure $\bar{F}$ of $F$, 
obtained as the compositum of splitting fields of degree $n$ polynomials with coefficients in $F$. Since there are only finitely many extensions of a local field of bounded degree, $E$ is a finite extension of $F$. In particular it is a local field. Hence, there are only finitely many roots of unity in $E$; let this cardinality be $m$. 
 
	We can assume that  $g_1$ and $g_2$ are semisimple. Let $\{\alpha_1,\cdots,\alpha_n\}$ (resp. $\{\beta_1,\cdots,\beta_n\}$) be the eigenvalues of $g_1$ (resp. $g_2$). Since these are roots of the characteristic polynomials of $g_1$ and $g_2$, they lie in $E$. By hypothesis,  up to a permutation 
	$$\alpha_i^k=\beta_i^k,\hspace{5mm}(\forall 1\leq i\leq n).$$
Hence $\alpha_i$ and $\beta_i$ differ by a root of unity, which lies in $E$. Thus from the above comment, we have 
	$\alpha_i^m=\beta_i^m$ for $ 1\leq i\leq n.$ 
	Therefore both $g_1^m$ and $g_2^m$ are conjugate in $GL_n(\bar{F})$, hence  $g_1^m$ and $g_2^m$ are conjugate in $GL_n(F)$.
\end{proof}

\subsection{Potential equivalence for $\ell$-adic representations}
We now consider potential equivalence in the context of $\ell$-adic representations of the absolute Galois group of number fields. Chebotarev-type density results allow us to work with Frobenius conjugacy classes in determining potential equivalence of representations.  

We first fix some notation.  For a place $v$ of $K$
let $K_v$ denote the completion of $K$ at $v$.  Choosing a place $w$
of $\bar{K}$ lying above $v$ allows us to identity $G_{K_v}$ with the
decomposition subgroup $D_w$ of $G_K$. As $w$ varies this gives a
conjugacy class of subgroups of $G_K$. Given a representation $\rho$
of $G_K$ as above, define the localization (or the local component) \(
\rho_v \) of $\rho$ at $v$, to be the representation of $G_{K_v}$
obtained by restricting $\rho$ to a decomposition subgroup. This is
well-defined up to isomorphism. 

Denote by $I_v$ the inertia subgroup of $G_{K_v}$. 
Let $F$ be a non-archimedean local field of characteristic zero and residue characteristic $\ell$. A continuous representation $\rho: G_K\to GL_n(F)$ is unramified at a finite place $v$ of $K$ if its localization $\rho_v$ is trivial restricted to an inertia group $I_v$ at $v$. 
We recall that the upper density of a set \( S \)
of finite places of \( K \) is defined as:
\[ 
ud (S) := \limsup_{x \to \infty } \# \{ v \in S | ~ N v \leq x \}/ \pi (x) , 
\] 
where \( \pi (x ) \) is the number of finite places \( v \) of \( K
\) with  \( N v \leq x \). Here  \( N v \) denotes the cardinality
of the residue field $k_v$ of \( K_v \). 
\begin{definition}
Suppose $\rho_1, \rho_2$ are representations of $G_K$ to $GL_n(F)$, and let $T$ be a set of places of $K$. The representations are said to be {\em locally potentially equivalent} at the set of places $T$ of $K$, if the restrictions 
 $\rho_{1,v}$ and $\rho_{2,v}$ to the local Galois groups $G_{K_v}$ are potentially equivalent at the places $v$ belonging to $T$. 
\end{definition}
 
Given a continuous 
representation  $\rho: ~G_K\to GL_n(F)$ and a place
$v $ of $K$ where $\rho$ is unramified, let $\rho(\sigma_v)$ denote
the Frobenius conjugacy class in the image group $  G_K/{\rm
Ker}(\rho) \simeq \rho( G_K) \subset GL_n(F)$.  By an abuse of
notation, we will also continue to denote by \( \rho (\sigma_v) \) an
element in the associated conjugacy class.

\begin{theorem}\label{theorem:poteq-mtraceeq}
Let $K$ be a number field and $F$ a non-archimedean local field of characteristic zero. Suppose  $\rho_1$ and $\rho_2$ are continuous, semisimple representations of $G_K$ to $GL_n(F)$. Then the following are equivalent: 
\begin{enumerate}
\item The representations $\rho_1$ and $\rho_2$ are potentially equivalent. 
\item There exists a set of places $T$ of $K$ of upper density one such that the local components $\rho_{1,v}$ and $\rho_{2,v}$ are potentially equivalent as representations of $G_{K_v}$. 
\item There exists a positive integer $m$ depending only on $n$ and $F$ such that at a set of places $T$ of upper density one of $K$ which are unramified for $\rho_1$ and $\rho_2$, 
\begin{equation}\label{eqn:assertion3}
 \mbox{Tr}~\rho_1(\sigma_v ^m)=\mbox{Tr}~\rho_2(\sigma_v^m ),\quad v\in T.
\end{equation}

\item The representations $\rho_1$ and $\rho_2$ are $m$-trace equivalent for some natural number $m$. 

\end{enumerate}

\end{theorem}
\begin{proof}
The implication $(1)\implies (2)$ is obvious, and $(4)\implies (1)$  is proved above. 

For the implication that $(2)\implies (3)$, it is proved in (\cite{KR}) that the collection of ramified places for any continuous semisimple representation of $G_K$ is of density zero. Hence, after removing a set of places of density zero, it can be assumed that at a set $T$ of finite places of $K$ that are unramified for $\rho_1$ and $\rho_2$ and having upper density one, the following holds true: 
for each $v\in T$, there is some natural number $k_v$ such that the conjugacy classes $\rho_1(\sigma_v^{k_v})$ and   $\rho_2(\sigma_v^{k_v})$ are conjugate in $GL_n(F)$. Part $(3)$ now follows from Theorem \ref{theorem:eltwisepe}. 

Hence it is left to prove $(3)\implies (4)$.   We first recall the following algebraic analogue of Chebotarev density theorem (\cite{Ra}), giving an algebraic interpretation of results proved in
Section 6 (especially Proposition 15) of (\cite{Se2}):

\begin{theorem}\label{Algebraic-Chebotarev}\cite[Theorem 3]{Ra}~~
Let \( M \) be an algebraic group defined over a $\ell$-adic local field
\( F \) of characteristic zero.  Suppose 
\[
\rho : G_K \rightarrow M(F) 
\]
is a continuous representation unramified
 outside a set $S$ of places of $K$ of density zero. 

Suppose  \( X \) is a closed subscheme of \( M\) 
defined over \( F \) and stable under the adjoint action of \( M \) 
on itself. Let $C := X(F) \cap \rho ( G_K ) .$ Let \( G_{\rho} \)  be
the algebraic  monodromy  group attached to \( \rho \) over \( F \),
i.e.,  the smallest algebraic subgroup $G_{\rho}$ of $M$ defined
over $F$ such that $\rho(G_K)\subset M(F)$.

Then the set $
T: =  \{ v \not\in S ~| ~ \rho ( \sigma_v ) \subset C \}$, 
has a density given by 
\[
d(T) = \frac{ | \Psi | }{ | \Phi | }, 
\]
where $\Phi$ is the set of connected components of $G_{\rho}$, and 
$\Psi $ is the set of those $\phi \in \Phi$ such that the
corresponding connected component $G_{\rho}^\phi$ of $G_{\rho}$ is contained in $X $.
\end{theorem}

\begin{remark} The foregoing theorem is proved in (\cite{Ra}), under the assumption that $S$ is finite. The proof goes through without any change upon assuming that $S$ is a set of places of density zero. The  results of (\cite{KR}) ensure that for any continuous semisimple representation of $G_K$, the set of ramified places is of density zero. This allows us to frame the hypothesis of Theorem \ref{Algebraic-Chebotarev} for general continuous semisimple representations rather than requiring that they be unramified outside a finite set $S$ of places of $K$. 

\end{remark}

Consider the subset of $GL_n\times GL_n$ consisting of elements with the same `$m$-trace':
\[X^{[m]}:=\{(g_1,g_2)\in GL_n\times GL_n~|~\mbox{Tr}~ (g_1^m) = \mbox{Tr}~ (g_2^m)\}.\]
The set $X^{[m]}$ is a conjugacy invariant, Zariski closed subvariety of $GL_n\times GL_n$. Let $G$ be the Zariski closure of the image $\rho(G_K):=(\rho_1\times \rho_2)(G_K)$ of $G_K$ in $GL_n\times GL_n$, and $G^0$ be the connected component of identity in $G$. From equation (\ref{eqn:assertion3}), at the places  $v\in T$, the corresponding pair 
$(\rho_1(\sigma_v), \rho_2(\sigma_v ))$ belongs to $X^{[m]}$. By Assertion $(3)$, the set of places $T$ has upper density one. Applying Theorem \ref{Algebraic-Chebotarev}, it follows that 
 $G\subset X^{[m]}$.  In particular, this yields Part $(4)$ of the theorem, assuming $(3)$. 

\end{proof}

\section{Finiteness criteria for potential equivalence}	
In \cite{Fa}, Faltings gave a finite criteria for the equivalence
of two $\ell$-adic representations of the absolute Galois
group $G_K$ of a number field $K$, which are 
unramified outside a fixed finite set $S$ of places of $K$: 
\begin{theorem}[Faltings]  \label{theorem:Faltings}
Let $K$ be a number field and $\ell$ a rational prime. Let $S$ be a finite set of non-archimedean
places of $K$ containing the places dividing $\ell$ of $K$, and $F$ be a non-archimedean local field of characteristic zero and residue characteristic $\ell$. 
Fix a natural number $n$. 

Then there exists 
a finite set $T$ of finite places of  $K$,
disjoint from $S$,  such that the isomorphism class of any continuous semi-simple representation $\rho:G_K\longrightarrow GL_n(F)$ unramified outside $S$, is determined uniquely by the finite collection $\{\text{Tr}~\rho(\sigma_v)\mid v\in T\}$, where 
$\rho(\sigma_v)$ denotes the Frobenius conjugacy class corresponding to the place $v$ in the group $\rho(G_K)$ and $\text{Tr}$ is the trace of a square matrix. 
\end{theorem}

We now extend Faltings' criteria 
to the context of $m$-power characters:

\begin{theorem}\label{theorem:main}
	Let $K$ be a number field. Fix a rational prime $\ell$ and natural numbers $n, a$ and $m$. Let  $S$ be a finite set of places of $K$ containing the archimedean places 
	and the places above $\ell$ 
	of $K$.
	
	Then there exists a finite set $T$ of finite places of $K$ disjoint from $S$, 
	with the following property: suppose  $\rho_1, \rho_2:G_K\longrightarrow GL_n(F)$ are two continuous  semisimple representations of $G_K$, where $F$ is a non-archimedean field with residue field of cardinality $\ell^a$. Assume that for $v$ not in $S$, the inertia at $v$ acts by scalar matrices of order dividing $m$, and that the representations satisfy  the following hypothesis: 
	$$\mbox{Tr}~\rho_1(\sigma_v^m )=\mbox{Tr}~\rho_2(\sigma_v^m ), \quad \text{for}~v\in T.$$
	Then $\rho_1$ and $\rho_2$ are potentially isomorphic. 
\end{theorem}
We make two remarks: suppose $w$ is a place of $\bar{K}$ lying above $v$ not in $S$. By hypothesis, the inertia at $w$ acts by a scalar. As we vary the place $w$ lying above $v$, the conjugacy class of Frobenius elements $\rho(\sigma_w )$ are well-defined only up to scalar matrices of order dividing $m$, where $m$ is chosen as in Lemma \ref{lemma}. However the conjugacy class $\rho(\sigma_v^m)$ is well-defined. 

Note also that the field $F$ is allowed to vary: only the cardinality of its residue field is fixed. For example, one can work with $F$, a totally ramified extension of $\Q_{\ell}$ without affecting the finiteness bounds. 

 The proof of this theorem follows that of Faltings' theorem, once the function $\chi^{[m]}$ is linearized using Newton's identities. The proof is given in the following steps:

\begin{enumerate}[wide, labelwidth=!, labelindent=0pt]
\item{\em Construction of $T$.}
Let $V=F^n$. For $1\leq i\leq n$, let  $\Lambda^iV$ denote the $i$-th exterior vector space associated to $V$. For any $n$-tuple of non-negative integers $\bar{r}:=(r_1,\cdots, r_n),$ consider the vector space, 
\[\Lambda^{\bar{r}}V:= V^{\otimes r_1}\otimes( \Lambda^2V)^{\otimes r_2}\otimes \cdots \otimes (\Lambda^nV)^{\otimes r_n}. \]
Given a representation $\rho$ of $G_K$ on $V$, there is an associated representation $\Lambda^{\bar{r}}\rho$ on the space $\Lambda^{\bar{r}}V$. Define the degree $d(\bar{r})$ of $\bar{r}$ as $d(\bar{r})=r_1+2r_2+\cdots+nr_n$. Denote by $\mbox{dim}(\bar{r}, V)$, the dimension of the space $\Lambda^{\bar{r}}V$. The hypothesis that the inertial action outside $S$ acts by scalar matrices of order dividing $m$, implies that the associated representations $\Lambda^{\bar{r}}\rho$ are unramified outside $S$ provided $d(\bar{r})=m$. Let
\[ d_m=2\left(\sum_{d(\bar{r})=m} \mbox{dim}(\bar{r}, V)^2\right).\]

It follows from the finiteness theorem of Hermite-Minkowski that there are only finitely many Galois extensions $L/K$ that are unramified outside $S$ and degree bounded by $\ell^{ad_m}$. Let $L^{[m]}$ denote the finite Galois extension of $K$ given by the compositum of these finite extensions. 

By Chebotarev density theorem, we choose a finite set $T$ of places of $K$ such that Frobenius conjugacy classes at $v$ of $L^{[m]}/K$ cover all the conjugacy classes of $G(L^{[m]}/K)$. \\

\item{\em Nakayama's lemma and finite generation.} 
Given a representation $\rho$ of $G_K$ acting on $V$, consider the associated representation $\rho^{[m]}$ of $G_K$ on $V^{[m]}$, where 
\[ V^{[m]}=\oplus_{d(\bar{r})=m}  \Lambda^{\bar{r}}V.\]

Suppose $\rho_1$ and $\rho_2$ are as in the hypothesis of the theorem. Let $V_1$ and $V_2$ be the underlying vector spaces on which these representations act. The assumption on the action of inertia outside $S$ implies that the representations  $\Lambda^{\bar{r}}\rho_1$ and $\Lambda^{\bar{r}}\rho_2$ are unramified outside $S$, where $d(\bar{r})=m$. Let ${\mathcal O}_F$ denote the ring of integers of $F$. 
Let $M$ be the ${\mathcal O}_F$ submodule of 
\[ \oplus_{d(\bar{r})=m}\mbox{End}(V_1^{\bar{r}})\oplus \mbox{End}(V_2^{\bar{r}}),\]
generated by the image of $G_K$.  Its rank over ${\mathcal O}_F$ is at most $d_m$. 

Let ${\mathfrak m}_F$ denote the maximal ideal in ${\mathcal O}_F$. The dimension of the quotient module $M/{\mathfrak m}_FM$ over the residue field $k_F$ is at most $d_m$. Thus the cardinality of the image of the Galois group $G_K$ in $(M/{\mathfrak m}_FM)^*$ is less than $\ell^{ad_m}$. Since  it is unramified outside $S$, it is a quotient of $G(L^{[m]}/K)$. 
By the construction of $T$, the image of the Frobenius conjugacy classes at $v\in T$ of the representation of $G_K$ on $M/{\mathfrak m}_FM$,   generate the quotient module $M/{\mathfrak m}_FM$. 

By Nakayama's lemma, it follows that the Frobenius conjugacy classes $(\rho_1^{[m]}\times \rho_2^{[m]})(\sigma_v)$ for $v \in T$ generate the module $M$. \\

\item{\em Newton's identity.} Newton's identities allow us to write the function $\chi_{\rho}^{[m]}(g)=\chi_{\rho}(g^m)$ as an integral linear combination of the characters  $\chi_{\Lambda^{\bar{r}}\rho}$  associated to the  representation $\Lambda^{\bar{r}}\rho$ of $G_K$. For any $m\geq 1$, Newton's identity for $g\in GL_n(F)$ gives, 
\[ \mbox{Tr}(g^m)=(-1)^mm\sum_{d(\bar{r})=m}\frac{(r_1+r_2+\cdots+r_n -1)!}{r_1!r_2!\cdots r_n!}\prod_i(-e_i)^{r_i},\]
where $e_i=\mbox{Tr}(\Lambda^i(g))$ is the $i$-th elementary symmetric function. Here we have used the fact that since the dimension is $n$, only the elementary symmetric functions with index at most $n$ contribute to the sum. In terms of characters, this yields the following expression for $\chi_{\rho}^{[m]}$:
\begin{equation}\label{eqn:Newton}
\chi_{\rho}^{[m]}(g)=(-1)^mm\sum_{d(\bar{r})=m}\frac{(r_1+r_2+\cdots+r_n-1)!}{r_1!r_2!\cdots r_n!}(-1)^{\sum r_i}\chi_{\Lambda^{\bar{r}}\rho}.
\end{equation} 
Since the collection of Frobenius conjugacy classes $(\rho_1^{[m]}\times \rho_2^{[m]})(\sigma_v)$ for $v\in T$ generate the ${\mathcal O}_F$-module $M$, they also generate the image of $M$ in the module $\mbox{End}(\Lambda^{\bar{r}}V_1)\oplus \mbox{End}(\Lambda^{\bar{r}}V_2)$ for $d(\bar{r})=m$. It follows that the character values $\chi_{\Lambda^{\bar{r}}\rho}$ are determined by their evaluations at this conjugacy classes. The Newton identity 
expresses the function $\chi_{\rho}^{[m]}$ as an integral linear combination of the characters $\chi_{\Lambda^{\bar{r}}\rho}$. Hence under the assumption of the theorem, we obtain
\[ \chi_{\rho_1}^{[m]}(g)=\chi_{\rho_2}^{[m]}(g)\quad \forall g\in G_K.\]

\item{\em Proof of Theorem \ref{theorem:main}.} The proof of Theorem \ref{theorem:main} is now completed by appealing to Theorem \ref{theorem:poteq-mtraceeq}.

\end{enumerate}

We now give an extension of  Faltings' criteria 
to the context of potential equivalence, but now we need to work with a fixed local field $F$: 

\begin{corollary}\label{cor:pe}
	Let $K$ be a number field. Fix a natural numbers $n$  and a non-archimedean local field $F$ of residue characteristic  $\ell$. 
	Let  $S$ be a finite set of places of $K$ containing the archimedean places 
	and the places above $\ell$ 
	of $K$.
	
	Then there exists a finite set $T$ of finite places of $K$ disjoint from $S$, 
	with the following property: suppose  $\rho_1, \rho_2:G_K\longrightarrow GL_n(F)$ are two continuous  semisimple representations of $G_K$. Assume that for $v$ not in $S$, the inertia at $v$ acts by scalar matrices, and that the representations  are locally potentially equivalent at the set of places $v\in T$.

	Then $\rho_1$ and $\rho_2$ are potentially isomorphic. 
\end{corollary}

Since the inertia acts by scalar matrices and the number of roots of unity in a non-archimedean local field is finite,  it follows that the inertia acts by finite order scalar matrices. The corollary follows now from choosing $m$ as in Lemma \ref{lemma} and appealing to the above theorem. 

\begin{remark} It would be interesting to extend the foregoing result to global fields, i.e., to function fields of a curve over a finite field. The main issue is to find suitable replacements for the Hermite-Minkowski finiteness theorem. 
We thank the referee for pointing out the difficulties involved in extending the results to the setting of function fields. 
\end{remark}

 \section{Purity and finiteness theorems for potentially  equivalent $\ell$-adic representations}

\subsection{Purity}
We recall that an algebraic integer $\alpha \in \bar{\Q}$ is said to be a  Weil number of weight $w$ (corresponding to a finite field $\F_q$ or a non-archimedean local field with residue field isomorphic to $\F_q$), if for any embedding $\phi: \bar{\Q}\to \C$, $|\phi(\alpha)|=q^{w/2}$. 
The representation $\rho:G_K\rightarrow GL_n(F)$ is said to be pure of
weight $w$ and degree at most $d$,  if at any finite place $v$ of $K$ where $\rho$ is unramified, the eigenvalues of the  Frobenius conjugacy class $\rho(\sigma_v)$ are algebraic integers of weight $w$ and degree at most $d$ over $\Q$. 

We obtain a finiteness result for {\em potential} equivalence classes of 
pure semi-simple $\ell$-adic global Galois representations. This generalizes 
Faltings' finiteness result for the equivalence classes of 
pure semi-simple $\ell$-adic global Galois representations.
\begin{theorem}\label{thm:finpe}
Fix natural numbers $n, ~ w $ and $d$. Let $K$ be a number field and $\ell$ a rational prime. Let $S$ be a finite set of non-archimedean
	places of $K$ containing the places dividing $\ell$, and $F$ be a non-archimedean local field of characteristic zero and residue characteristic $\ell$. 
	
	Then there are, up to potential equivalence, only finitely many continuous, semi-simple Galois representations $\rho: G_{K} \rightarrow
	GL_n(F)$, such that the inertia at places outside $S$ acts by a scalar matrix,  and 
	 	 pure of degree at most $d$ and weight  $w$.
\end{theorem}
\begin{proof}
Since the inertia at a place $v\not\in S$ acts by scalars valued in a non-archimedean local field $F$, it is a central element of finite order. Let $m$ be as in Lemma \ref{lemma}. As in the remark following the statement of Theorem \ref{theorem:main},  the Frobenius conjugacy class $\rho(\sigma_v^m)$ is well-defined. By Corollary \ref{cor:pe}, there exists a finite set of places $T$ of $K$ disjoint from $S$, such that potential equivalence of the Frobenius classes at these places determines the potential equivalence of the representations. It can be seen that there are only finitely many Weil numbers attached to $\F_q$ of degree bounded by $d$ and weight $w$. Consequently, there are only finitely many choices for the roots of the characteristic polynomial of $\rho(\sigma_v^m)$ for $v\in T$, and hence Theorem \ref{thm:finpe} follows. 
\end{proof}
Similarly, one can derive a finiteness result using $m$-power characters with fixed $m$, but varying $F$, as in hypothesis of Theorem \ref{theorem:main}: 
\begin{corollary}\label{cor:finpe-mchar}
	Fix natural numbers $n, ~\ell,  ~a, ~m,~ w $ and $d$. Let $K$ be a number field and $\ell$ a rational prime. Let $S$ be a finite set of non-archimedean
	places of $K$ containing the places dividing $\ell$, and $F$ be a non-archimedean local field of characteristic zero and residue field of cardinality $\ell^a$. 
	
	Then there are, up to potential equivalence, only finitely many continuous, semi-simple Galois representations $\rho: G_{K} \rightarrow
	GL_n(F)$, such that the inertia at places outside $S$ acts by a scalar matrix of order dividing $m$, and 
	 	 pure of degree at most $d$ and weight  $w$.
\end{corollary}

\subsection{Abelian varieties}
Suppose $A$ is an abelian variety defined over a number field $K$. Denote by $T_{\ell}(A)=\varprojlim_nA[\ell^n]$ the $\ell$-adic Tate module attached to $A$, where $A[\ell^n]$ denotes the group of elements of order  $\ell^n$ in $A(\bar{K})$. Let $V_\ell(A):=\Q_\ell\otimes_{\Z_\ell} T_{\ell}(A)$. 

It is shown in (\cite[Theorem 3]{ST}), that  at a place $v$ where $A$ has potential good reduction, the characteristic polynomial of the Frobenius conjugacy class $\rho(\sigma_v)$ is well-defined and has coefficients belonging to rational integers.  Since the weight is one, the assumptions of Theorem  \ref{thm:finpe} are satisfied. Consequently, the following holds: 
 
\begin{corollary}\label{cor:fin-abvar}
Let $K$ be a number field and $S$ be a finite set of non-archimedean
	places of $K$. 
	
	Then there are, up to potential equivalence, only finitely many Abelian varieties of dimension $g$ defined over $K$, such that the inertia at a place $v$ not in $S$ acts by scalars on a Tate module  $V_\ell(A)$ for $\ell$ coprime to the residue characteristic at $v$. 
\end{corollary}
\begin{proof} The hypothesis implies that the inertia acts by finite order scalars. Hence the abelian variety has potential good reduction at places $v$ not in $S$. By (\cite[Theorem 2]{ST}), the inertial character takes  values in $\Z$ and is independent of $\ell$. This limits the choice of the inertial character at places outside $S$ to be quadratic. It is also known that the representation on the Tate module is semi-simple. We can take $m=2$ in the foregoing corollary to obtain that up to potential equivalence, there are only finitely many Galois representations attached to Tate modules of abelian varieties satisfying the hypothesis of the corollary. 
The corollary follows as a consequence of the isogeny theorem of Tate-Zarhin-Faltings (\cite{Fa}), 
\[ \text{Hom}(A,B)\otimes \Q_{\ell}\simeq  \text{Hom}_{G_K}(V_{\ell}(A), V_{\ell}(B)),\]
where $A$ and $B$ are abelian varieties defined over $K$ together with the finiteness theorem of Faltings which says that each isogeny class of abelian varieties over the number field $K$ only contains finitely many isomorphism classes.
\end{proof}

\section{Twist unramified representations}
In this section, we consider finiteness theorems for twist unramified representations. 
\subsection{Forms and Dirichlet twists}
\begin{definition} Let $\rho: G_K\rightarrow GL_n(F)$ be a continuous representation of $G_K$. A {\em form} $\rho'$ of $\rho$ is  a continuous representation  $\rho': G_K\rightarrow GL_n(F)$ such that the restrictions of $\rho$ and $\rho'$ to $G_L$ are isomorphic, for some  finite extension $L$ of $K$. 

The representation $\rho: G_K\rightarrow GL_n(F)$ is \textit{twist unramified}  at a finite place $v$ of $K$ if there exists a form $\rho'$ of $\rho$ (possibly depending on $v$) which  is unramified at $v$. 
\end{definition}

We consider a more restricted notion of twisting by a finite order character.  Recall that a Dirichlet character of $K$ is a continuous character $\chi:G_K\to \bar{\Q}_{\ell}^\times$  of finite order.  

\begin{definition} Let $K$ be a number field and $S$ a set of places of $K$ containing the archimedean places of $K$. A continuous representation $\rho:G_K\to GL_n(F)$ is said to be {\em unramified outside $S$ up to Dirichlet twists}, if for each finite place $v$ not in $S$, there exists a Dirichlet character $\chi$ such that $\rho\otimes \chi$ is unramified at $v$. 

We will say that $\rho$ is {\em unramified outside $S$ up to $F$-valued Dirichlet twists} if in addition to the above, the character $\chi$ takes values in $F^*$. 
\end{definition}
A consequence of Hilbert Theorem 90 and Schur's lemma give a connection between forms and Dirichlet twists: 
\begin{proposition}\label{prop:Schur}
Suppose $\rho, ~\rho'$ are continuous representations of $G_K$  to $GL_n(F)$.
Suppose there is a finite extension $L$ of $K$ such that the restrictions of $\rho$ and $\rho'$ are absolutely irreducible and isomorphic. Then there exists a 
Dirichlet character $\chi$ of $G_K$ with values in $F^*$ such that $\rho'\simeq
 \rho\otimes \chi$. 
\end{proposition}
Note that the Dirichlet characters are actually valued in $F^*$. 
\begin{proof} The collection of matrices which conjugate the representation $\rho|_{G_L}$ to $\rho'|_{G_L}$ is a torsor for $GL_1$ defined over
$F$. By Hilbert Theorem 90, we can assume after conjugating by a matrix in $GL_n(F)$ that $\rho|_{G_L}\equiv \rho'|_{G_L}$. 

It can be checked that for every $\sigma\in G_K$, the operator $T(\sigma):=\rho'(\sigma)^{-1}\rho(\sigma)$ intertwines the representations $\rho|_{G_L}$ and $\rho'|_{G_L}$. By Schur's  lemma, $T(\sigma)=\chi(\sigma)I_n$, for some character $\chi:G_K\to \bar{F}^*$. Since $T(\sigma)$ is in $GL_n(F)$, this implies that $\chi$ takes values in $F^*$. 
\end{proof}

\subsection{Finiteness theorems for twists}
The following theorem is a consequence of Theorem \ref{thm:finpe}. 
\begin{theorem}\label{theorem:Dirichlet-pe}
	Let $K$ be a number field and $\ell$ a rational prime. Let $S$ be a finite set of non-archimedean
	places of $K$ containing the places dividing $\ell$ of $K$, and $F$ be a non-archimedean local field of characteristic zero and residue characteristic $\ell$. 
	Fix natural numbers $n, w$ and $d$. Then, up to potential equivalence,   there are only finitely many continuous semi-simple representations $\rho:G_K\longrightarrow\text{GL}_n(F)$ satisfying the following: 
\begin{itemize}
\item The representations $\rho$ are unramified outside $S$ up to Dirichlet twists. 
\item $\rho$ is  pure of weight at most $w$ and degree bounded by $d$. 

\end{itemize}
\end{theorem} 

From  the foregoing theorem  and Proposition \ref{prop:Schur}, we deduce now an analogue of Faltings' finiteness theorem for twist unramified representation: 
\begin{theorem}\label{theorem:fintwist}
	Let $K$ be a number field and $\ell$ a rational prime.  Let $S$ be a finite set of non-archimedean
	places of $K$ containing the places dividing $\ell$ of $K$, and $F$ be a non-archimedean local field of characteristic zero and residue characteristic $\ell$. 
	Fix natural numbers $n, w$ and $d$. 
	
Then, up to twist by finite order characters $\chi:G_K\rightarrow F^{*}$,  there are only finitely many isomorphism classes of continuous semi-simple representations $\rho:G_K\longrightarrow\text{GL}_n(F)$ satisfying the following: 
\begin{itemize}
\item The representations $\rho$ are twist unramified outside $S$. 
\item $\rho$ is  pure of weight at most $w$ and degree bounded by $d$. 
\item The connected component 
$G_\rho^{\circ}$ of identity of the algebraic monodromy group $G_{\rho}$ of $\rho$  acts absolutely irreducibly on $F^n$. 

\end{itemize}
\end{theorem} 
As an application of theorem $\ref{theorem:fintwist}$, we get finiteness of the set of isomorphism classes of non-CM elliptic curves over a number field with twisted good reduction outside $S$. 
\begin{corollary}\label{cor1}
	Let $S$ be a finite set of primes of $K$ containing the archimedean primes of $K$. The set of isomorphism class of non-CM elliptic curves defined over $K$ with twisted good reduction outside $S$ is finite up to twists.  
\end{corollary}
For non-CM elliptic curves, the associated Galois representation on its Tate module $\rho: G_K\rightarrow GL_2(\Z_\ell)$ has open image and is an absolutely irreducible representation. Hence it satisfies the hypothesis of Theorem $\ref{theorem:fintwist}$ with weight $w=1$. The corollary follows as in the proof of Corollary \ref{cor:fin-abvar}. 
 Note that corollary $\ref{cor1}$ is already proved in \cite[Theorem 0.3]{BB}. They adapt Shafarevich's well-known proof (see \cite[Chapter IV, Theorem 1.4]{Se1}) of the corresponding assertions for usual good reduction outside $S$. 

\begin{remark} Theorem \ref{theorem:fintwist} can be applied for the Galois representations arising from the Tate module of abelian varieties, provided the algebraic monodromy group is `sufficiently large' in that and it satisfies the hypothesis of the theorem. The theorem will be applicable for general members of modular families of abelian varieties whose Mumford-Tate group satisfies the hypothesis of Theorem \ref{theorem:fintwist}. 

From this, one can deduce a finiteness theorem for twist equivalence of higher genus curves, provided the Mumford-Tate group of the Hodge structure associated to the first cohomology is large enough. 

It is not in general true that a curve has good reduction if its Jacobian variety has good reduction. Oda has shown in (\cite[Theorem 3.2]{Od}), that a smooth, projective and geometrically connected  curve $C$ over $K$ has good reduction if and only if the outer $G_K$-action on its $\Q_\ell$-unipotent fundamental group $\pi_1^\et(C_{\overline{K}})_{\Q_\ell}$ is unramified for all $\ell\neq p$ (equivalently for some $\ell\ne p$). 
\end{remark}

\subsection*{Acknowledgements} The first author thanks the School of Mathematics, Tata Institute of Fundamental Research, Mumbai and UM-DAE Centre for Excellence in Basic Sciences, Mumbai for their support. The research of the first author is supported by a DST-INSPIRE Fellowship IF150349. 

We thank the referee for many valuable suggestions improving and correcting the exposition of the paper.

 \end{document}